\documentclass[reqno]{amsart}

\usepackage{amssymb,amscd,enumerate,mathrsfs,tikz} 
\usepackage{tikz}
\usepackage{siunitx}

\numberwithin{equation}{section}
\newtheorem{theorem}[equation]{Theorem}
\newtheorem{proposition}[equation]{Proposition}
\newtheorem{lemma}[equation]{Lemma}
\theoremstyle{definition}

\newtheorem{remark}[equation]{Remark}
\newtheorem{example}[equation]{Example}

\def\N{\mathbb N}
\def\R{\mathbb R}
\def\C{\mathbb C}

\def\tol{\textup{tol}_0}

\DeclareMathOperator{\arcoth}{arcoth}
\DeclareMathOperator{\arccot}{arccot}

\begin{document}

\title[Third-order method based on quadratic Taylor approximation]{An explicit third-order one-step method for autonomous scalar initial value problems of first order based on adaptive quadratic Taylor approximation}

\author{Thomas Krainer}
\address{Penn State Altoona\\ 3000 Ivyside Park \\ Altoona, PA 16601-3760}
\email{krainer@psu.edu}

\author{Chenzhang Zhou}
\address{Penn State Altoona\\ 3000 Ivyside Park \\ Altoona, PA 16601-3760}
\email{cjz5145@psu.edu}

\begin{abstract}
We present an explicit $1$-step numerical method of third order that is error-free on autonomous scalar Riccati equations such as the logistic equation. The method replaces the differential equation by its quadratic Taylor polynomial in each step and utilizes the exact solution of that equation for the calculation of the next approximation.
\end{abstract}


\maketitle


\section{Introduction}

\noindent
One of the basic ordinary differential equations in quantitative population dynamics is the logistic differential equation
$$
\left\{
\begin{aligned}
\dot{y} &= ry\Bigl(1-\frac{y}{K}\Bigr) - qy, \\
y\big|_{t=0} &= y_0,
\end{aligned}
\right.
$$
where $y=y(t)$ is the size of the population at time $t \geq 0$, $r > 0$ is the maximal growth rate for the population, $K > 0$ the carrying capacity of the habitat for the population under study. We modified the equation in this example by a harvesting term with harvesting rate $q > 0$ as it appears, for example, in fishery models. We refer to \cite{Thieme} as a general reference for ordinary differential equations models in ecology. While the logistic model, as well as its variations and perturbations, are classical cornerstones of ecological quantitative modeling, it is remarkable that the standard numerical methods for approximating solutions to ordinary differential equations do not solve the logistic equation error-free. Motivated by this observation, we are presenting here an explicit third-order $1$-step numerical method that is applicable to scalar autonomous initial value problems of the form
\begin{equation}\label{GenericODE}
\left\{
\begin{aligned}
\dot{y} &= f(y) \\
y\big|_{t=0} &= y_0
\end{aligned}
\right.
\end{equation}
with sufficiently smooth real-valued $f$ that approximates the solution $y=y(t)$ on the compact interval $[0,T]$ by a sequence of values $y_0,y_1,y_2,\ldots$ based on equidistant time-stepping with step size $h > 0$, and whose distinguishing feature is that the method is error-free if $f$ is a polynomial up to degree two such as in the logistic equation, i.e., our method solves autonomous Riccati equations exactly. The idea for this method is simple:
\begin{enumerate}
\item Replace $f(y)$ in \eqref{GenericODE} by its quadratic Taylor polynomial $T_{y_0}(y)$ centered at $y_0$.
\item Solve $\dot{u} = T_{y_0}(u)$ exactly with initial condition $u(0) = y_0$.
\item Set $y_1 = u(h)$ and repeat with $y_1$ in place of $y_0$, etc.
\end{enumerate}
There are several issues that arise upon implementation of this basic idea. Most importantly, it must be noted that solutions to Riccati equations can blow up in finite time, so integrity checks on the step size $h > 0$ are needed to preclude a potential blow-up of the approximate solution $u$ on the interval $(0,h]$ as otherwise the calculated term $y_1$ is invalid (and likewise in subsequent steps). To make this more transparent, consider the following example:
$$
\left\{
\begin{aligned}
\dot{y} &= (y-\lambda)(1-y)e^{-y^4}, \\
y\big|_{t=0} &= 0,
\end{aligned}
\right.
$$
where $\lambda \gg 1$. The maximal solution to this differential equation exists on $(-\infty,\infty)$ because the function
$$
f(y) = (y-\lambda)(1-y)e^{-y^4}, \; -\infty < y < \infty,
$$
is bounded. However, the maximal solution to $\dot{u} = T_{0}(u) = (u-\lambda)(1-u)$ with initial value $u(0) = 0$ blows up at $t = \frac{\ln(\lambda)}{\lambda - 1}$. In view of $\lim\limits_{\lambda \to \infty}\frac{\ln(\lambda)}{\lambda - 1} = 0$ we see that blow-up does occur on $(0,h)$ if $\lambda \gg 1$ is large enough.

We provide two options to deal with this problem in the implementation, a priori or at run time. The a priori option calculates a threshold for how small the step size $h > 0$ ought to be chosen at the outset to avoid invalid approximating terms throughout and is based on the differential equation \eqref{GenericODE} and the viewing window $[0,T]\times[y_{\min},y_{\max}]$ where its solution is supposed to be approximated as inputs, while the run time option checks validity of each approximation value at the time when it is calculated. The issue of blow-up is germane to the method we discuss in this note, it does not occur in standard Runge-Kutta methods or exponential integrators.

A second issue that we needed to address in the implementation concerns evaluation of the formula for the approximate solution $u$ itself. If the roots of the Taylor polynomial $T_{y_0}$ are distinct but close, the exact formula for $u$ would require evaluation of quotients nearly of the form $\frac{0}{0}$ (but with defined limiting value corresponding to the double-root case). We deal with this by introducing a tolerance parameter $0 < \tol \ll 1$ and replace evaluation of the exact formula for $u$ by appropriate expansions once the critical expressions fall under tolerable thresholds. Problems of similar kind are well-known to arise elsewhere in numerical ODEs, for example in exponential integrators where evaluation of $\phi_1(z) = \frac{e^z-1}{z}$ for $z$ near zero occurs, see \cite{HochbruckOstermann,KassamTrefethen}.

\vspace*{1em}

\noindent
The general idea of utilizing zeroth- and first-order Taylor approximations in the differential equation is well-established both in theoretical and computational ODEs. In computational ODEs, adaptive first-order Taylor approximation (linearization) is the basis for exponential integrators, classically rooted in the Rosen\-brock-Euler method (observe that adaptive zeroth-order Taylor approximation in the differential equation yields the Euler method). Exponential integrators \cite{HochbruckOstermann} have been widely used for stiff problems over the past 30 years; they are generally more effective for these problems than standard Runge-Kutta methods because the linearization of the differential equation is solved exactly. Since the theoretical underpinning for exponential integrators is linear theory, they have been developed into a versatile family of methods applicable to single equations and systems alike. The method we present in this note based on adaptive quadratic Taylor approximation of the differential equation is qualitatively more accurate than methods rooted in linearization, but does not exhibit the same degree of versatility and universality, and the applicability is strictly limited to autonomous scalar equations. The reason is that ordinary differential equations with quadratic nonlinearities generally do not allow for closed solution formulas, the autonomous case of a single unknown function being an exception.

We note that our work relates to nonstandard finite difference models and their applications to numerical ODEs as pioneered by Mickens \cite{Mickens94,Mickens00}, see also \cite{Patidar}. In particular, exact nonstandard finite difference models for the logistic equation and many other ODEs where explicit solution formulas are available are well-known \cite{VigoAguiarRamos}.

\vspace*{1em}

\noindent
The paper is structured as follows: Section~\ref{SectionTaylorMethod} covers the theoretical part. We prove, more generally than what has been stated above, that when the function $f$ in \eqref{GenericODE} is adaptively replaced by its $r$-th order Taylor polynomial and the exact solution to the modified ODE is used to calculate the next approximating value, we obtain a well-defined convergent explicit numerical method of order $r+1$. More precisely, when the exact solution is supposed to be approximated in the window $[0,T]\times[y_{\min},y_{\max}]$, we show that there is a threshold $h_0 > 0$ such that the method is defined everywhere in that window for step sizes $0 < h < h_0$ and allows calculation of the next approximating value to the solution. This qualitatively addresses the aforementioned blow-up issue (that is not present for $r=0$ and $r=1$ of course). The proofs utilize some results about ODEs depending on parameters and an abstract theorem about the convergence of $1$-step methods, stated in the needed forms in Appendices~\ref{ConvergenceOneStepMethods} and \ref{ODEsDependingOnParameters}.

Section~\ref{QuadraticTaylor} contains the core of this paper. We discuss the formulas of the method based on quadratic Taylor approximation and their adjustments based on the aforementioned tolerance considerations, the quantitative a priori as well as run time aspects of step size control to address the blow-up issue, and discuss in detail the numerical algorithms. The MATLAB code of the programs is listed in Appendix~\ref{SourceCode}.

Section~\ref{TestsQuadraticTaylor} contains the results of numerical tests of the method, using MATLAB, with benchmarks against some Runge-Kutta methods of orders $3$ and $4$, respectively. We have tested the quadratic Taylor method on some standard equations from population dynamics, in line with our original motivation, as well as other equations. Our results on the tested equations confirm that the method based on quadratic Taylor expansion can fare better on the global error by several orders of magnitude when compared to the tested Runge-Kutta methods.

\vspace*{1em}

\noindent
As was mentioned before, we only consider equidistant time-stepping in this paper. We also do not utilize any extrapolation techniques to further improve our method. There are certainly several avenues of investigation, in parallel to established ones for standard numerical methods, that could be pursued to augment the method presented in this paper and improve it further. However, the fact that general Riccati equations do not allow for closed solution formulas is going to remain a limiting factor.


\section{Convergence of explicit methods based on exactly solving Taylor approximations of the differential equation}\label{SectionTaylorMethod}

\noindent
Let $r \in {\mathbb N}_0$, and let $f : D \to \R$ be $(r+1)$-times continuously differentiable on the open set $D \subset \R$, and suppose $y : [0,T] \to D$ solves the initial value problem
\begin{equation}\label{IVP}
\left\{
\begin{aligned}
\dot{y}(t) &= f(y(t)) \textup{ on } 0 \leq t \leq T, \\
y\big|_{t=0} &= y_0 \in D.
\end{aligned}
\right.
\end{equation}
As mentioned in the introduction, and idea for an explicit method is to locally replace $f$ by its $r$-th order Taylor polynomial and take the exact solution of the resulting differential equation with the Taylor polynomial instead of $f$ as numerical approximation for $y : [0,T] \to D$ over small time steps. To pursue this idea, define $F : \R\times D \to \R$ via
\begin{equation}\label{TaylorF}
F(w,y) = \sum\limits_{j=0}^{r} \frac{f^{(j)}(y)}{j!}w^j,
\end{equation}
and let $w(h,y)$ for $(h,y) \in U_{\max}$ be the maximally extended solution of
$$
\left\{
\begin{aligned}
\frac{\partial w}{\partial h}(h,y) &= F(w(h,y),y), \\
w(0,y) &= 0.
\end{aligned}
\right.
$$
This differential equation for $w$ depends on $y$ as a parameter, and we have summarized some results about differential equations depending on parameters that we will use below in Appendix~\ref{ODEsDependingOnParameters}. Since $F$ and all its partial $w$-derivatives are continuously differentiable with respect to $(w,y)$ in $\R\times D$, we obtain that $\partial_h^k w$ is continuously differentiable with respect to $(h,y) \in U_{\max}$ for all $k \in \N_0$. Define $\Phi : U_{\max} \to \R$ via 
\begin{equation}\label{TaylorMethod}
\Phi(h,y) = y + w(h,y).
\end{equation}
Observe that $\Phi$ solves
$$
\left\{
\begin{aligned}
\frac{\partial \Phi}{\partial h}(h,y) &= \sum\limits_{j=0}^{r} \frac{f^{(j)}(y)}{j!}\bigl(\Phi(h,y)-y\bigr)^j \\
\Phi(0,y) &= y.
\end{aligned}
\right.
$$

\begin{proposition}\label{TaylorPhiProperties}
$\Phi$ and all its partial $h$-derivatives are continuously differentiable with respect to $(h,y) \in U_{\max}$. For every compact subset $K \Subset D$ there exists $h_0 > 0$ such that $\Phi : [0,h_0] \times K \to \R$ is defined, and $\frac{\partial \Phi}{\partial h} : [0,h_0] \times K \to \R$ satisfies a Lipschitz condition with respect to $y$ in $[0,h_0] \times K$.
\end{proposition}
\begin{proof}
By Theorem~\ref{ODEParameters} and Remark~\ref{RemarkODEParameters}, $w$ and all its partial $h$-derivatives exist and are continuously differentiable on $U_{\max}$, and for every $K \Subset D$ there exists $h_0 > 0$ such that  $w : [0,h_0] \times K \to \R$ is defined. All this is therefore also true for $\Phi$. Since $\frac{\partial^2\Phi}{\partial y\partial h} : U_{\max} \to \R$ exists and is continuous, $\frac{\partial\Phi}{\partial h} : [0,h_0] \times K \to \R$ satisfies a Lipschitz condition with respect to $y$ as claimed.\end{proof}

\begin{proposition}[Local Truncation Error]\label{LocalTruncationError}
For any compact neighborhood $K \Subset D$ with $y([0,T]) \subset \mathring{K}$ there exist $h_0 > 0$ such that $\Phi : [0,h_0] \times K \to \R$ is defined, and a constant $C \geq 0$ independent of $0 \leq h \leq h_0$ and $0 \leq t \leq T$ such that
\begin{equation*}
\bigl|y(t+h) - \Phi(h,y(t))\bigr| \leq C h^{r+2}
\end{equation*}
whenever $0 \leq t+h \leq T$.

If $f$ is a polynomial of degree $\leq r$ we have $y(t+h) = \Phi(h,y(t))$, i.e., the method is locally exact.
\end{proposition}
\begin{proof}
Recall that if $u$ and $v$ are $n$-times differentiable, $n \in \N$, Fa{\`a} di Bruno's formula asserts that
$$
\frac{d^n}{dh^n}(u\circ v)(h) = \sum\limits_{k=1}^n u^{(k)}(v(h))B_{n,k}(v'(h),v''(h),\ldots,v^{(\mu^n_k)}(h))
$$
with the partial Bell polynomials
$$
B_{n,k}(x_1,\ldots,x_{\mu^n_k}) = \sum_{\substack{\alpha\in\N_0^{\mu^n_k}, \;|\alpha|=k \\ 1\alpha_1+2\alpha_2+\ldots+\mu^n_k\alpha_{\mu^n_k} = n}} \frac{n!}{\alpha!} \cdot \Bigl(\frac{x_1}{1!}\Bigr)^{\alpha_1}\Bigl(\frac{x_2}{2!}\Bigr)^{\alpha_2}\cdots\Bigl(\frac{x_{\mu^n_k}}{\mu^n_k!}\Bigr)^{\alpha_{\mu^n_k}},
$$
where $\mu^n_k = n-k+1$. We now proceed to use this formula in order to show inductively that
\begin{equation}\label{Taylorcoeff}
\frac{d^n}{dh^n}y(t+h)\Big|_{h=0} = \frac{\partial^n}{\partial h^n}\Phi(h,y(t))\Big|_{h=0}
\end{equation}
for $n = 0,\ldots,r+1$ (note that $y \in C^{r+2}([0,T])$ since $f \in C^{r+1}(D)$ by assumption). For $n=0$ this follows immediately from the definition in \eqref{TaylorMethod}, keeping in mind that $w(0,y) = 0$. For $n = 1$ we have
\begin{align*}
\frac{d}{dh}y(t+h)\Big|_{h=0} &= f\bigl(y(t+h)\bigr)\Big|_{h=0} = f(y(t)), \\
\frac{\partial}{\partial h}\Phi(h,y(t))\Big|_{h=0} &= F(w(h,y(t)),y(t))\Big|_{h=0} = F(0,y(t)) = f(y(t)).
\end{align*}
So suppose we know \eqref{Taylorcoeff} for all $n \leq n_0$ for some $1 \leq n_0 \leq r$. Now
\begin{align*}
\frac{d^{n_0+1}}{dh^{n_0+1}}y(t+h) &= \frac{d^{n_0}}{dh^{n_0}}\Bigl(\frac{d}{dh}y(t+h)\Bigr) = \frac{d^{n_0}}{dh^{n_0}}(f\circ y)(t+h) \\
&= \sum\limits_{k=1}^{n_0} f^{(k)}(y(t+h))B_{n_0,k}(y'(t+h),y''(t+h),\ldots,y^{(\mu^{n_0}_k)}(t+h)).
\end{align*}
Evaluation at $h=0$ gives
\begin{equation}\label{n01y}
\frac{d^{n_0+1}}{dh^{n_0+1}}y(t+h)\Big|_{h=0} = \sum\limits_{k=1}^{n_0} f^{(k)}(y(t))B_{n_0,k}(y'(t),y''(t),\ldots,y^{(\mu^{n_0}_k)}(t)).
\end{equation}
Using the differential equation for $w(h,y)$ and \eqref{TaylorMethod} we get
$$
\frac{\partial^{n_0+1}}{\partial h^{n_0+1}}\Phi(h,y(t))\Big|_{h=0} = \frac{\partial^{n_0}}{\partial h^{n_0}}F(w(h,y(t)),y(t))\Big|_{h=0},
$$
which by Fa{\`a} di Bruno's formula equals
\begin{equation}\label{n01Phi}
\sum\limits_{k=1}^{n_0} \bigl(\partial_w^kF\bigr)(w(0,y(t)),y(t))B_{n_0,k}\bigl((\partial_hw)(0,y(t)),\ldots,(\partial_h^{\mu^{n_0}_k}w)(0,y(t))\bigr).
\end{equation}
By induction, $(\partial_h^jw)(0,y(t)) = y^{(j)}(t)$ for $j=1,\ldots,\mu^{n_0}_k$, and thus the arguments in the partial Bell polynomials $B_{n_0,k}$ in \eqref{n01y} and \eqref{n01Phi} agree. Moreover,
$$
\bigl(\partial_w^kF\bigr)(w(0,y(t)),y(t)) = \bigl(\partial_w^kF\bigr)(0,y(t))= f^{(k)}(y(t))
$$
for $k=0,\ldots,r$ in view of \eqref{TaylorF} and Taylor's formula. This shows that \eqref{Taylorcoeff} holds for $n=n_0+1$ and finishes the induction.

In view of \eqref{Taylorcoeff}, Taylor's formula now implies
\begin{gather*}
\bigl|y(t+h) - \Phi(h,y(t))\bigr| = \Bigl|\int_0^h \frac{y^{(r+2)}(t+s)-\bigl(\partial_h^{r+2}\Phi\bigr)(s,y(t))}{(r+1)!}(h-s)^{r+1}\,ds\Bigr| \\
\leq \underbrace{\Bigl[\frac{1}{(r+2)!}\Bigl(\max\limits_{0\leq s \leq T}|y^{(r+2)}(s)| + \max\limits_{\substack{0 \leq s \leq h_0 \\ y \in K}}|(\partial_h^{r+2}\Phi)(s,y)|\Bigr)\Bigr]}_{=:C} \cdot h^{r+2}
\end{gather*}
for all $0 \leq t \leq T$ and $0 \leq h \leq h_0$ such that $t+h \leq T$.

If $f$ is a polynomial of degree $\leq r$ we have
$$
f(z) = \sum\limits_{j=0}^r \frac{f^{(j)}(y)}{j!}(z-y)^j
$$
for all $z,y \in \R$, and consequently both $h \mapsto y(t+h)$ and $h \mapsto \Phi(h,y(t))$ solve
$$
\left\{
\begin{aligned}
\dot{u}(h) &= f(u(h)), \quad h \geq 0, \\
u\big|_{h=0} &= y(t).
\end{aligned}
\right.
$$
By uniqueness we must therefore have $y(t+h) = \Phi(h,y(t))$.
\end{proof}

\begin{theorem}\label{TaylorConvergent}
The method $\Phi$ defined in \eqref{TaylorMethod} is a convergent method of order $r+1$ for the approximation of the solution $y : [0,T] \to \R$ of \eqref{IVP}. The method is exact for differential equations \eqref{IVP} when $f$ is a polynomial of degree at most $r$.
\end{theorem}
\begin{proof}
This follows with Propositions~\ref{TaylorPhiProperties} and \ref{LocalTruncationError} from Theorem~\ref{ConvergenceTheoremGeneral}.
\end{proof}

\begin{example}
If we specialize to $r=0$ and $r=1$ in \eqref{TaylorF} we find familiar methods.
\begin{itemize}
\item If $r=0$ we have $F(w,h) = f(y)$ in \eqref{TaylorF}, and so $\Phi(h,y)$ solves
$$
\left\{
\begin{aligned}
\frac{\partial \Phi}{\partial h}(h,y) &= f(y) \\
\Phi(0,y) &= y.
\end{aligned}
\right.
$$
Thus $\Phi(h,y) = y + hf(y)$ is the Euler method.
\item If $r=1$ we have $F(w,y) = f(y) + f'(y)w$ in \eqref{TaylorF}, and so $\Phi(h,y)$ solves
$$
\left\{
\begin{aligned}
\frac{\partial \Phi}{\partial h}(h,y) &= f(y) + f'(y)\bigl(\Phi(h,y)-y) \\
\Phi(0,y) &= y.
\end{aligned}
\right.
$$
Thus $\Phi(h,y) = y + h\phi_1\bigl(f'(y)h\bigr)f(y)$ with
\begin{equation*}
\phi_1(z) = \begin{cases}
\dfrac{e^z - 1}{z} &\textup{for } z \neq 0 \\
1 &\textup{for }z = 0
\end{cases}
\end{equation*}
is the Rosenbrock-Euler method \cite[Section~2.4]{HochbruckOstermann}.
\end{itemize}
\end{example}

In this paper, we present and analyze the method based on adaptive Taylor approximation in detail for $r = 2$. While the theoretical result in Theorem~\ref{TaylorConvergent} holds for all $r \in {\mathbb N}_0$, it is not feasible for the implementation of methods for larger $r$ as one generally does not have explicit solution formulas for polynomial ordinary differential equations.


\section{Third order scheme based on quadratic Taylor approximation}\label{QuadraticTaylor}

\noindent
We begin by defining and analyzing the analytic function
\begin{equation}\label{psiuv}
\psi(u,v) = \frac{\sinh(v)}{u\sinh(v) + v\cosh(v)}
\end{equation}
depending on two complex variables $(u,v) \in \C^2$. Initially, this function is undefined on
$$
S = \{(u,v) \in \C^2;\; u\sinh(v) + v\cosh(v) = 0\}.
$$
It is easy to see that $S$ consists of the complex line $v = 0$ and the complex surface
$$
S_{\psi} : u = -v\coth(v).
$$
The singularities of $\psi$ where $v = 0$ are removable, except for the singularity at the single branch point $(-1,0)$. To see this note that for $v$ near $0$ we can write
$$
\psi(u,v) = \frac{1}{u + v\coth(v)}.
$$
The function $v \mapsto v\coth(v)$ has a removable singularity at $v = 0$. The first few terms of the Taylor series are
$$
v\coth(v) = 1 + \frac{v^2}{3} - \frac{v^4}{45} + O(v^6),
$$
which shows that $\psi$ has a removable singularity at all points $(u,0)$ if $u \neq -1$. More precisely, we get
\begin{equation}\label{psiexpansion}
\psi(u,v) = \frac{1}{1+u} - \frac{v^2}{3(1+u)^2} + \frac{(u+6)v^4}{45(1+u)^3} + O(v^6)
\end{equation}
locally uniformly in $u$, and, in particular, we see that the definition
\begin{equation}\label{psiline}
\psi(u,0) = \frac{1}{1+u}, \quad u \neq -1,
\end{equation}
extends $\psi$ analytically to $v=0$ except the branch point. From \eqref{psiuv} and \eqref{psiline} we obtain that $\psi$ remains singular on $S_{\psi} : u = -v\coth(v)$, but now with the understanding that the singularity of $v\coth(v)$ when $v = 0$ has been removed. We are going to need the function $\psi$ only in the cases that both $u$ and $v$ are real, or that $u$ is real and $v$ is imaginary. Figure~\ref{SingularSet} shows parts of the intersection of the singular set $S_{\psi}$ with ${\mathbb R}^2$ and with ${\mathbb R}\times i{\mathbb R}$, respectively.

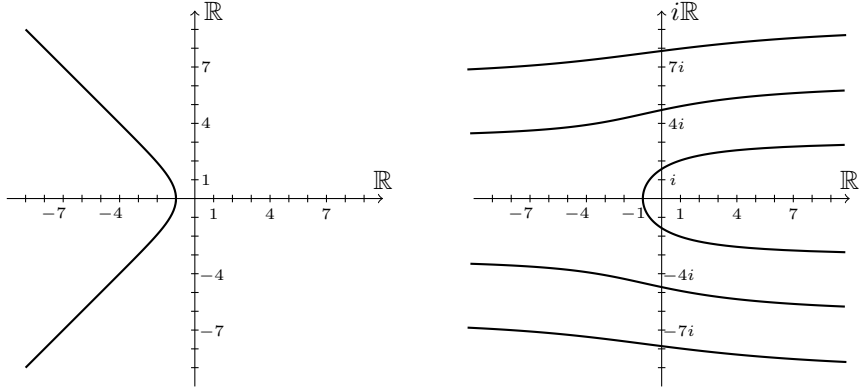
\begin{figure}[ht]
\begin{minipage}{6cm}
\begin{tikzpicture}[scale=0.25]
\draw[->] (-10,0) -- (10,0) node[above] {${\mathbb R}$};
\draw[->] (0,-10) -- (0,10) node[right] {${\mathbb R}$};
\foreach \x in {-9,-8,-7,-6,-5,-4,-3,-2,-1,1,2,3,4,5,6,7,8,9}
	{
	\draw (\x,-0.2) -- (\x,0.2);
	\draw (-0.2,\x) -- (0.2,\x);
	}
\node at (1,-.8) {\tiny $1$};
\node at (4,-.8) {\tiny $4$};
\node at (7,-.8) {\tiny $7$};
\node at (-4.5,-.8) {\tiny $-4$};
\node at (-7.5,-.8) {\tiny $-7$};
\node at (0.6,1) {\tiny $1$};
\node at (0.6,4) {\tiny $4$};
\node at (0.6,7) {\tiny $7$};
\node at (0.9,-4.1) {\tiny $-4$};
\node at (0.9,-7.1) {\tiny $-7$};
\draw[domain=-9:9,thick,samples=100] plot ({-(1)*(\x)*(cosh((\x))/sinh((\x)))},\x);
\end{tikzpicture}
\end{minipage}
\begin{minipage}{6cm}
\begin{tikzpicture}[scale=0.25]
\draw[->] (-10,0) -- (10,0) node[above] {${\mathbb R}$};
\draw[->] (0,-10) -- (0,10) node[right] {$i{\mathbb R}$};
\foreach \x in {-9,-8,-7,-6,-5,-4,-3,-2,-1,1,2,3,4,5,6,7,8,9}
	{
	\draw (\x,-0.2) -- (\x,0.2);
	\draw (-0.2,\x) -- (0.2,\x);
	}
\node at (1,-.8) {\tiny $1$};
\node at (4,-.8) {\tiny $4$};
\node at (7,-.8) {\tiny $7$};
\node at (-1.5,-.8) {\tiny $-1$};
\node at (-4.5,-.8) {\tiny $-4$};
\node at (-7.5,-.8) {\tiny $-7$};
\node at (0.6,1) {\tiny $i$};
\node at (0.8,4) {\tiny $4i$};
\node at (0.8,7) {\tiny $7i$};
\node at (0.9,-4.1) {\tiny $-4i$};
\node at (0.9,-7.1) {\tiny $-7i$};
\draw[domain=-2.857:2.857,thick,samples=100] plot ({-(1)*(\x)*cot((\x) r)},\x);
\draw[domain=3.47:5.75,thick,samples=100] plot ({-(1)*(\x)*cot((\x) r)},\x);
\draw[domain=6.87:8.7,thick,samples=100] plot ({-(1)*(\x)*cot((\x) r)},\x);
\draw[domain=-5.75:-3.47,thick,samples=100] plot ({-(1)*(\x)*cot((\x) r)},\x);
\draw[domain=-8.7:-6.87,thick,samples=100] plot ({-(1)*(\x)*cot((\x) r)},\x);
\end{tikzpicture}
\end{minipage}
\caption{Singularities of $\psi$: $S_{\psi} \cap {\mathbb R}^2$ (left) and $S_{\psi} \cap \bigl({\mathbb R}\times i{\mathbb R}\bigr)$ (right)}\label{SingularSet}
\end{figure}

The relevance of the function $\psi$ for us is clarified by the following lemma. The proof is straightforward and will be omitted.

\begin{lemma}\label{RiccatiSolution}
Consider the initial value problem for the Riccati ordinary differential equation
\begin{equation}\label{RiccatiConstant}
\left\{
\begin{gathered}
\dot{w} = aw^2 + bw + c \\
w\big|_{t=0} = 0
\end{gathered}
\right.
\end{equation}
with constant coefficients $a,b,c \in {\mathbb R}$. Let $\Delta = b^2 - 4ac$, and define $\alpha = -\frac{b}{2} \in {\mathbb R}$, $\beta = \frac{\sqrt{\Delta}}{2} \in {\mathbb C}$. Note that $\beta \in {\mathbb R}$ if $\Delta \geq 0$, and in case $\Delta < 0$ we choose $\sqrt{\Delta}$ to be the root with positive imaginary part\footnote{We could choose either, really, since $\psi(u,v)$ is even in $v$.}, so $\beta \in i{\mathbb R}_+$.

The maximal solution to \eqref{RiccatiConstant} is given by
$$
w(t) \equiv w(t;a,b,c) = ct\psi(t\alpha,t\beta), \quad t_{\min} < t < t_{\max},
$$
where
\begin{align*}
t_{\min} &= \sup\{t < 0;\; (t\alpha,t\beta) \in S_{\psi}\} \in {\mathbb R}_- \cup \{-\infty\}, \\
t_{\max} &= \inf\{t > 0;\; (t\alpha,t\beta) \in S_{\psi}\} \in {\mathbb R}_+ \cup \{\infty\}.
\end{align*}
\end{lemma}

It is important to note that the solution $w(t)$ to \eqref{RiccatiConstant} can blow up in finite time depending on the values of the constants $a,b,c$. This has a serious impact on the method presented here in that additional integrity checks on the step size must be performed (a priori or at run time) that do not appear in Runge-Kutta methods or exponential integrators. The way the solution $w(t)$ is represented in Lemma~\ref{RiccatiSolution} utilizing the function $\psi$ facilitates a simple visualization of the existence interval. As described in the lemma, the coefficients $a,b,c \in {\mathbb R}$ determine a point $(\alpha,\beta) \in {\mathbb R}^2$, or $(\alpha,\beta) \in {\mathbb R}\times i{\mathbb R}$, respectively, and the solution involves evaluation of the function $\psi$ restricted to the line passing through the origin in ${\mathbb R}^2$ (or ${\mathbb R}\times i{\mathbb R}$) and that point, parametrized by $t$. $t_{\max}$ is precisely the first positive $t$-value when that line crosses the singular set $S_{\psi}$ of the function $\psi$ (and $t_{\max} = \infty$ if there is no such crossing point), and similarly for $t_{\min}$. It is easy to visualize this behavior in Figure~\ref{SingularSet}. We summarize, focussing on $t_{\max}$:
\begin{itemize}
\item $\Delta \geq 0$, so $(\alpha,\beta) \in {\mathbb R}^2$: $t_{\max} < \infty$ precisely when $\alpha < 0$ and $|\beta| < |\alpha|$. These conditions are equivalent to $b > 0$ and $0 < ac \leq \frac{b^2}{4}$. By definition of $S_{\psi}$, $t_{\max}$ is then the (unique) solution to
$$
\alpha t_{\max} = -\beta t_{\max}\coth(\beta t_{\max})
$$
(recall that $v\coth(v) = 1$ when $v = 0$). Consequently,
$$
t_{\max} =
\begin{cases}
-\frac{1}{\alpha} = \frac{2}{b}  & \textup{if } \Delta = 0, \\
\frac{1}{\beta}\arcoth(-\frac{\alpha}{\beta}) =  \frac{1}{\sqrt{\Delta}}\ln\bigl(\frac{b+\sqrt{\Delta}}{b-\sqrt{\Delta}}\bigr) & \textup{if } \Delta > 0.
\end{cases}
$$
\item If $\Delta < 0$ we have $t_{\max} < \infty$, and
$$
\alpha t_{\max} = -\beta t_{\max}\coth(\beta t_{\max}) = -(-i\beta)t_{\max}\cot(-i\beta t_{\max}).
$$
We get\footnote{In the formula for $t_{\max}$ and elsewhere we use the real $\arccot : {\mathbb R} \to (0,\pi)$.}
$$
t_{\max} = \frac{1}{(-i\beta)}\arccot\Bigl(-\frac{\alpha}{(-i\beta)}\Bigr) = \frac{2}{\sqrt{-\Delta}}\arccot\Bigl(\frac{b}{\sqrt{-\Delta}}\Bigr).
$$
\end{itemize}
In summary,
$$
t_{\max} =
\begin{cases}
\frac{2}{b}  & \textup{if } \Delta = 0,\; b > 0, \\
\frac{1}{\sqrt{\Delta}}\ln\bigl(\frac{b+\sqrt{\Delta}}{b-\sqrt{\Delta}}\bigr) & \textup{if } \Delta > 0, \; \sqrt{\Delta} < b, \\
\frac{2}{\sqrt{-\Delta}}\arccot\Bigl(\frac{b}{\sqrt{-\Delta}}\Bigr) & \textup{if } \Delta < 0, \\
\infty & \textup{otherwise}.
\end{cases}
$$
In either case, since the closest point of $S_{\psi} \cap {\mathbb R}^2$, or $S_{\psi} \cap \bigl({\mathbb R} \times i{\mathbb R}\bigr)$, respectively, to the origin with respect to the Euclidean distance $|(\cdot,\cdot)|$ is the point $(-1,0)$, we note that the solution is guaranteed to exist while $|(t \alpha, t \beta)| < 1$. This gives a rough estimate
\begin{equation}\label{Aprioritmax}
t_{\max} \geq \frac{1}{|(\alpha,\beta)|} = \frac{2}{\sqrt{b^2 + |\Delta|}},
\end{equation}
where we understand the right-hand side of \eqref{Aprioritmax} to be $\infty$ if both $b=\Delta=0$. This estimate can be used to derive an a priori estimate on valid step sizes of the method, as described below.

\subsection*{Description of the method}

The goal is to approximate the solution $y = y(t)$ to the initial value problem
$$
\left\{
\begin{aligned}
\dot{y} &= f(y) \\
y\big|_{t=0} &= y_0
\end{aligned}
\right.
$$
for $(t,y) \in [0,T] \times [A,B]$. We assume that $f$ is $C^3$ in an open neighborhood of $[A,B]$, and $y_0 \in [A,B]$. To fix notation, define functions $a,b,c,\Delta,h_{\max} : [A,B] \to {\mathbb R} \cup \{\infty\}$ via
\begin{gather*}
a(y) = \frac{f''(y)}{2}, \quad b(y) = f'(y), \quad c(y) = f(y), \quad \Delta = b^2 - 4ac, \\
h_{\max}(y) =
\begin{cases}
\frac{2}{b(y)}  & \textup{if } \Delta(y) = 0,\; b(y) > 0, \\
\frac{1}{\sqrt{\Delta(y)}}\ln\Bigl(\frac{b(y)+\sqrt{\Delta(y)}}{b(y)-\sqrt{\Delta(y)}}\Bigr) & \textup{if } \Delta(y) > 0, \; \sqrt{\Delta(y)} < b(y), \\
\frac{2}{\sqrt{-\Delta(y)}}\arccot\Bigl(\frac{b(y)}{\sqrt{-\Delta(y)}}\Bigr) & \textup{if } \Delta(y) < 0, \\
\infty & \textup{otherwise}.
\end{cases}
\end{gather*}
Choose a tolerance $0 < \tol \ll 1$. Quantities that in absolute value are less than $\tol$ are considered numerically zero. Evaluation of approximate $\frac{0}{0}$ expressions with denominators of magnitude $< \tol$, such as occur in the evaluation of $\psi(u,v)$ given by \eqref{psiuv} for $v$ near zero, should be avoided to improve stability. For this reason, we define
\begin{equation}\label{MethodPhiDef}
\Phi(h,y) =
\begin{cases}
y + \frac{2c(y)\sinh\bigl[\frac{\sqrt{\Delta(y)}}{2}h\bigr]}{\sqrt{\Delta(y)}\cosh\bigl[\frac{\sqrt{\Delta(y)}}{2}h\bigr] - b(y)\sinh\bigl[\frac{\sqrt{\Delta(y)}}{2}h\bigr]} & \textup{for } (h,y) \in U_+, \\
y + \frac{2c(y)\sin\bigl[\frac{\sqrt{-\Delta(y)}}{2}h\bigr]}{\sqrt{-\Delta(y)}\cos\bigl[\frac{\sqrt{-\Delta(y)}}{2}h\bigr] - b(y)\sin\bigl[\frac{\sqrt{-\Delta(y)}}{2}h\bigr]} & \textup{for } (h,y) \in U_-, \\
y + \frac{2c(y)h}{2-b(y)h} - \frac{h^3c(y)\Delta(y)}{3(2-b(y)h)^2} & \textup{for } (h,y) \in U_0,
\end{cases}
\end{equation}
where
\begin{align*}
U_+ &= \{(h,y) \in [0,\infty)\times[A,B];\; \Delta(y) \geq 4\,\tol,\; h < h_{\max}(y),\; 2-hb(y) \geq \sqrt{\tol}\}, \\
U_- &= \{(h,y) \in [0,\infty)\times[A,B];\; \Delta(y) \leq -4\,\tol,\; h < h_{\max}(y),\; 2-hb(y) \geq \sqrt{\tol}\}, \\
U_0 &= \{(h,y) \in [0,\infty)\times[A,B];\; |\Delta(y)| < 4\,\tol,\; 2-hb(y) \geq \sqrt{\tol}\}.
\end{align*}
Some comments are in order:
\begin{itemize}
\item By Lemma~\ref{RiccatiSolution} and formula \eqref{psiuv}, the first two cases in \eqref{MethodPhiDef} are the exact formulas of the general method \eqref{TaylorMethod} discussed in Section~\ref{SectionTaylorMethod} with $r = 2$. In the second case we merely converted to trigonometric functions in the formulas since the argument of the hyperbolic trigonometric functions would be imaginary.
\item Taylor expansion of the hyperbolic trigonometric functions in the denominator in the first case gives
\begin{align*}
&\sqrt{\Delta(y)}\cosh\bigl[\tfrac{\sqrt{\Delta(y)}}{2}h\bigr] - b(y)\sinh\bigl[\tfrac{\sqrt{\Delta(y)}}{2}h\bigr] \\
=
&\sqrt{\Delta(y)}\sum\limits_{k=0}^{\infty}\tfrac{\bigl[\frac{\sqrt{\Delta(y)}}{2}h\bigr]^{2k}}{(2k)!} - b(y)\sum\limits_{k=0}^{\infty}\tfrac{\bigl[\frac{\sqrt{\Delta(y)}}{2}h\bigr]^{2k+1}}{(2k+1)!} \\
=
&\tfrac{\sqrt{\Delta(y)}}{2}\sum\limits_{k=0}^{\infty}\tfrac{1}{(2k)!}\bigl[2 - \tfrac{hb(y)}{2k+1}\bigr] \bigl[\tfrac{\sqrt{\Delta(y)}}{2}h\bigr]^{2k} \\
\geq &\,\tol \cdot \cosh\bigl[\tfrac{\sqrt{\Delta(y)}}{2}h\bigr] \geq \tol
\end{align*}
under the assumption that both $\frac{\sqrt{\Delta(y)}}{2} \geq \sqrt{\tol}$ and $2 - hb(y) \geq \sqrt{\tol}$, which explains the definition of $U_+$.
\item In the second case, we note that when $b(y) > 0$ we have $h_{\max}(y) < \frac{2}{b(y)}$ directly from the definition when $\Delta(y) < 0$, and consequently $2 - hb(y) > 0$ is implied by $h < h_{\max}(y)$ (the inequality is trivially fulfilled for $b(y) \leq 0$). The condition $2 - hb(y) \geq \sqrt{\tol}$ gives an extra buffer. We also note, arguing analogous to the first case, that
\begin{align*}
&\sqrt{-\Delta(y)}\cos\bigl[\tfrac{\sqrt{-\Delta(y)}}{2}h\bigr] - b(y)\sin\bigl[\tfrac{\sqrt{-\Delta(y)}}{2}h\bigr] \\
=
&\tfrac{\sqrt{-\Delta(y)}}{2}\sum\limits_{k=0}^{\infty}\tfrac{(-1)^k}{(2k)!}\bigl[2 - \tfrac{hb(y)}{2k+1}\bigr] \bigl[\tfrac{\sqrt{-\Delta(y)}}{2}h\bigr]^{2k} \\
&= \tfrac{\sqrt{-\Delta(y)}}{2}\Bigl[2 - hb(y) + O\Bigl[\bigl[\tfrac{\sqrt{-\Delta(y)}}{2}h\bigr]^2\Bigr] \Bigr],
\end{align*}
and thus the denominator is asymptotically $\geq \tol$ under the restrictions placed on $U_-$.
\item If the discriminant term $\Delta(y)$ is too small, evaluation of $\psi(u,v)$ as given by \eqref{psiuv} is unstable, so we opt to use the expansion \eqref{psiexpansion} instead for such terms, leading to the definition of $\Phi(h,y)$ in the third case. By Lemma~\ref{RiccatiSolution} and expansion \eqref{psiexpansion}, we note that the theoretical method as determined by \eqref{TaylorMethod} and our definition for $\Phi(h,y)$ in the third case of \eqref{MethodPhiDef} coincide to third order in $h$ as $h \to 0$, showing that the local truncation error in our definition is still $O(h^4)$ as required. Moreover, our definition for $\Phi(h,y)$ in the third case matches the general method from \eqref{TaylorMethod} if $\Delta(y) = 0$.
\end{itemize}

\subsubsection*{The algorithm}

Besides the differential equation $\dot{y} = f(y)$ and the initial value $y_0$, the inputs are $0 < \tol \ll 1$, the window $[0,T]\times[A,B]$ where the solution $y(t)$ is supposed to be approximated, and the chosen step size $h > 0$ for constructing an approximating sequence $y_0,y_1,\ldots$ of values for the solution at equidistant points $t = jh$, $j=0,1,\ldots$

\begin{enumerate}[(1)]
\item Check whether $y_0 \in [A,B]$. If not, the algorithm terminates with an error message that the initial value lies outside of the chosen tracking window.
\end{enumerate}

\noindent
Now suppose that an approximating partial sequence $y_0,\ldots,y_n$ for some $n \in {\mathbb N}_0$ has already been successfully constructed.

\begin{enumerate}[(1)]
\setcounter{enumi}{1}
\item If $(n+1)h > T$, the algorithm terminates with success and displays the approximation $y_0,\ldots,y_n$ of the solution.
\item \emph{Integrity check on the step size}\/: Check whether $(h,y_n) \in U_+ \cup U_- \cup U_0$. If not, the program terminates with the message that the algorithm stops after $n$ steps, approximating the solution on $[0,nh]$, as the method becomes undefined in the next step due to the chosen step size. The approximation of the solution thus far is displayed, and it is suggested to run the program again with a smaller step size $h > 0$.
\item Check whether $\Phi(h,y_n) \in [A,B]$. If not, the program is terminated with the message that the algorithm stops after $n$ steps, approximating the solution on $[0,nh]$, as the approximate solution is leaving the designated tracking window in the next step. The approximation of the solution thus far is displayed.
\item If the program reaches this step, it accepts $y_{n+1} = \Phi(h,y_n)$ as the next value of the approximating sequence, and recursively resumes at step (2) with $n$ incremented by one.
\end{enumerate}

Instead of performing the integrity check on the step size in (3) at run time during every execution of the recursive loop, an a priori estimate can be obtained prior to building the approximating sequence to determine a value $h_0 > 0$ that only depends on $f$, $\tol$, and the chosen viewing window such that all step sizes $0 < h < h_0$ work. Following this procedure and skipping the integrity checks at run time increases the speed of the program. The a priori estimate utilizes \eqref{Aprioritmax}, as follows:

\begin{enumerate}[(i)]
\item Find the maximum value $b_{\max}$ of $b : [A,B] \to {\mathbb R}$.
\item Find the maximum value $s_{\max}$ of $s:= b^2 + |\Delta| : [A,B] \to {\mathbb R}$.
\item Set
$$
h_0 =
\begin{cases}
\min\bigl\{\frac{2}{\sqrt{s_{\max}}},\frac{2-\sqrt{\tol}}{b_{\max}},T\bigr\} &\textup{if } s_{\max} > \tol,\; b_{\max} > \tol, \\
\min\bigl\{\frac{2}{\sqrt{s_{\max}}},T\bigr\} &\textup{if } s_{\max} > \tol,\; b_{\max} \leq \tol, \\
T &\textup{otherwise.}
\end{cases}
$$
\end{enumerate}


\section{Numerical tests of the quadratic Taylor method}\label{TestsQuadraticTaylor}

\noindent
In all tests described below we used the tolerance $\tol = \num{1E-14}$ and recorded the global error of the method on the indicated interval for the problem with various step sizes $h$. Errors in magnitude less than $\tol$ have been recorded as zero. All tests were performed using MATLAB. We are benchmarking our third order method, labeled QT3 below, against the following standard methods from the Runge-Kutta family:

\begin{itemize}

\item Kutta third order method (K3), see \cite[Section 233]{Butcher}: The Butcher tableau for this method is

\begin{center}
\begin{tabular}{c|ccc}
$0$ & \\[1.2ex]
$\frac{1}{2}$ & $\frac{1}{2}$ \\[1.2ex]
$1$ & $-1$ & $2$ \\[1.2ex] \hline
& $\frac{1}{6}$ & $\frac{2}{3}$ & $\frac{1}{6}$ 
\end{tabular}
\end{center}

\item Bogacki-Shampine third order method (BS3), see \cite{BogackiShampine}: The Butcher tableau for this method is

\begin{center}
\begin{tabular}{c|cccc}
$0$ & \\[1.2ex]
$\frac{1}{2}$ & $\frac{1}{2}$ \\[1.2ex]
$\frac{3}{4}$ & $0$ & $\frac{3}{4}$ \\[1.2ex]
$1$ & $\frac{2}{9}$ & $\frac{1}{3}$ & $\frac{4}{9}$ \\[1.2ex] \hline
& $\frac{2}{9}$ & $\frac{1}{3}$ & $\frac{4}{9}$ & $0$
\end{tabular}
\end{center}

\noindent
Embedded in a 3(2) pair, this method is built into one of the standard algorithms, {\ttfamily ode23}, of the MATLAB suite \cite{ShampineReichelt}. In \cite{BogackiShampine} the third order formulas are credited to Ralston \cite{Ralston}.

\item Classical Runge-Kutta fourth order method (RK4), see \cite[Section II.1]{HairerNorsettWanner}: The Butcher tableau for this method is

\begin{center}
\begin{tabular}{c|cccc}
$0$ & \\[1.2ex]
$\frac{1}{2}$ & $\frac{1}{2}$ \\[1.2ex]
$\frac{1}{2}$ & $0$ & $\frac{1}{2}$ \\[1.2ex]
$1$ & $0$ & $0$ & $1$ \\[1.2ex] \hline
& $\frac{1}{6}$ & $\frac{1}{3}$ & $\frac{1}{3}$ & $\frac{1}{6}$
\end{tabular}
\end{center}

\end{itemize}

\subsection*{Logistic equation}

As expected, the quadratic Taylor method outperforms standard Runge-Kutta methods for quadratic ordinary differential equations. Consider
$$
\left\{
\begin{aligned}
\dot{y} &= y(10-y) \\
y\big|_{t=0} &= 0.5 
\end{aligned}
\right.
$$
on the interval $[0,2]$. The exact solution is
$$
y(t) = \frac{10e^{10t}}{19+e^{10t}}, \quad 0 \leq t \leq 2.
$$
\begin{center}
\begin{tabular}{|c|c|c|c|c|}
\hline $h$ &  K3 & BS3 & RK4 & QT3 \\ \hline
$0.1$ & $\num{9.0574E-02}$ & $\num{4.9747E-02}$ & $\num{1.3532E-02}$ & $0$ \\ \hline
$0.05$ & $\num{1.3495E-02}$ & $\num{8.2625E-03}$ & $\num{1.0941E-03}$ & $0$ \\ \hline
$0.02$ & $\num{9.6842E-04}$ & $\num{6.3000E-04}$ & $\num{3.3012E-05}$ & $0$ \\ \hline
$0.01$ & $\num{1.2579E-04}$ & $\num{8.3520E-05}$ & $\num{2.1834E-06}$ & $0$ \\ \hline
\end{tabular}
\end{center}

\subsection*{Bernoulli equation}

Consider
$$
\left\{
\begin{aligned}
\dot{y} &= y\Bigl(1 - \Bigl(\frac{y}{20}\Bigr)^2\Bigr) \\
y\big|_{t=0} &= \num{1E-04}
\end{aligned}
\right.
$$
on the interval $[0,5]$. The exact solution is
$$
y(t) = \frac{20}{\sqrt{(\num{4E10}-1)e^{-2t} + 1}}, \quad 0 \leq t \leq 5.
$$
\begin{center}
\begin{tabular}{|c|c|c|c|c|}
\hline $h$ &  K3 & BS3 & RK4 & QT3 \\ \hline
$0.1$ & $\num{2.8543E-06}$ & $\num{2.8543E-06}$ & $\num{5.6900E-08}$ & $\num{9.6127E-13}$ \\ \hline
$0.05$ & $\num{3.7135E-07}$ & $\num{3.7135E-07}$ & $\num{3.7073E-09}$ & $\num{1.2390E-13}$ \\ \hline
$0.02$ & $\num{2.4343E-08}$ & $\num{2.4343E-08}$ & $\num{9.7307E-11}$ & $0$ \\ \hline
$0.01$ & $\num{3.0673E-09}$ & $\num{3.0673E-09}$ & $\num{6.1326E-12}$ & $0$ \\ \hline
\end{tabular}
\end{center}

\noindent
Let's also consider the same differential equation
$$
\left\{
\begin{aligned}
\dot{y} &= y\Bigl(1 - \Bigl(\frac{y}{20}\Bigr)^2\Bigr) \\
y\big|_{t=0} &= 1
\end{aligned}
\right.
$$
on the same interval $[0,5]$, but with a different initial value that is farther away from the equilibrium solutions. The exact solution is then
$$
y(t) = \frac{20}{\sqrt{399e^{-2t}+1}}, \quad 0 \leq t \leq 5.
$$
\begin{center}
\begin{tabular}{|c|c|c|c|c|}
\hline $h$ &  K3 & BS3 & RK4 & QT3 \\ \hline
$0.1$ & $\num{6.3817E-04}$ & $\num{4.5295E-04}$ & $\num{1.5055e-05}$ & $\num{3.2525E-04}$ \\ \hline
$0.05$ & $\num{8.1554E-05}$ & $\num{5.8683E-05}$ & $\num{9.2633e-07}$ & $\num{4.1018E-05}$ \\ \hline
$0.02$ & $\num{5.2845E-06}$ & $\num{3.8374E-06}$ & $\num{2.3554e-08}$ & $\num{2.6396E-06}$ \\ \hline
$0.01$ & $\num{6.6341E-07}$ & $\num{4.8314E-07}$ & $\num{1.4695e-09}$ & $\num{3.3052E-07}$ \\ \hline
\end{tabular}
\end{center}

\subsection*{Gompertz equation}

Consider
$$
\left\{
\begin{aligned}
\dot{y} &= y\ln\Bigl(\frac{30}{y}\Bigr) \\
y\big|_{t=0} &= 29
\end{aligned}
\right.
$$
on the interval $[0,2]$. The exact solution is
$$
y(t) = 30\Bigl(\frac{29}{30}\Bigr)^{e^{-t}}, \quad 0 \leq t \leq 2.
$$
\begin{center}
\begin{tabular}{|c|c|c|c|c|}
\hline $h$ &  K3 & BS3 & RK4 & QT3 \\ \hline
$0.1$ & $\num{1.5931E-05}$ & $\num{1.5604E-05}$ & $\num{3.1690E-07}$ & $\num{9.7263E-09}$ \\ \hline
$0.05$ & $\num{1.9169E-06}$ & $\num{1.8770E-06}$ & $\num{1.9019E-08}$ & $\num{1.1837E-09}$ \\ \hline
$0.02$ & $\num{1.1990E-07}$ & $\num{1.1734E-07}$ & $\num{4.7509E-10}$ & $\num{7.4419E-11}$ \\ \hline
$0.01$ & $\num{1.4873E-08}$ & $\num{1.4554E-08}$ & $\num{2.9431E-11}$ & $\num{9.2619E-12}$ \\ \hline
\end{tabular}
\end{center}

\subsection*{Flame propagation}

The following example is taken from a \emph{Cleve's Corner} blog post on the MathWorks web page, see \cite{Moler2003}. It is attributed there to L.~Shampine. Consider
$$
\left\{
\begin{aligned}
\dot{y} &= y^2 - y^3 \\
y\big|_{t=0} &= 0.98
\end{aligned}
\right.
$$
on the interval $[0,10]$. The exact solution is
$$
y(t) = \frac{1}{1 + W\bigl(\frac{1}{49}e^{\frac{1}{49}-t}\bigr)}, \quad 0 \leq t \leq 10,
$$
where $W$ is the Lambert W function, see \cite{Corlessetal}.
\begin{center}
\begin{tabular}{|c|c|c|c|c|}
\hline $h$ &  K3 & BS3 & RK4 & QT3 \\ \hline
$0.1$ & $\num{3.0134E-07}$ & $\num{2.8743E-07}$ & $\num{5.9219E-09}$ & $\num{3.8462E-10}$ \\ \hline
$0.05$ & $\num{3.6318E-08}$ & $\num{3.4589E-08}$ & $\num{3.5555E-10}$ & $\num{4.6768E-11}$ \\ \hline
$0.02$ & $\num{2.2745E-09}$ & $\num{2.1638E-09}$ & $\num{8.8861E-12}$ & $\num{2.9453E-12}$ \\ \hline
$0.01$ & $\num{2.8224E-10}$ & $\num{2.6843E-10}$ & $\num{5.5067E-13}$ & $\num{3.6637E-13}$ \\ \hline
\end{tabular}
\end{center}

\subsection*{An equation involving a sine function}

The following initial value problem is qualitatively similar to the logistic equation as well. Consider
$$
\left\{
\begin{aligned}
\dot{y} &= \sin(y) \\
y\big|_{t=0} &= 0.01
\end{aligned}
\right.
$$
on the interval $[0,1]$. The exact solution is
$$
y(t) = 2\arctan\bigl(\tan(0.005)e^t\bigr), \quad 0 \leq t \leq 1.
$$
\begin{center}
\begin{tabular}{|c|c|c|c|c|}
\hline $h$ & K3 & BS3 & RK4 & QT3 \\ \hline
$0.1$ & $\num{1.0453E-06}$ & $\num{1.0450E-06}$ & $\num{2.0837E-08}$ & $\num{3.4029E-10}$ \\ \hline
$0.05$ & $\num{1.3599E-07}$ & $\num{1.3594E-07}$ & $\num{1.3576E-09}$ & $\num{4.3857E-11}$ \\ \hline
$0.02$ & $\num{8.9142E-09}$ & $\num{8.9111E-09}$ & $\num{3.5634E-11}$ & $\num{2.8583E-12}$ \\ \hline
$0.01$ & $\num{1.1232E-09}$ & $\num{1.1228E-09}$ & $\num{2.2457E-12}$ & $\num{3.5945E-13}$ \\ \hline
\end{tabular}
\end{center}

\subsection*{Conclusion}

In the tested cases, the global error of our third order QT3 method is comparable and often smaller by several orders of magnitude than the global error of the other tested methods of the same order from the Runge-Kutta family. We even observed it to be smaller or comparable to the global error of the classical Runge-Kutta method of order four in most cases. This effect is most pronounced near equilibrium solutions of the tested equations.


\begin{appendix}


\section{Convergence of 1-step methods}\label{ConvergenceOneStepMethods}

\noindent
Let $D \subset \R$ be open, and suppose $f : D \to \R$ satisfies a local Lipschitz condition in $D$. Let $y : [0,T] \to D$ be the solution to the initial value problem
$$
\left\{\begin{aligned}
\dot{y}(t) &= f(y(t)) \textup{ on } 0 \leq t \leq T, \\
y\big|_{t=0} &= y_0 \in D.
\end{aligned}
\right.
$$
Theorem~\ref{ConvergenceTheoremGeneral} below is a general convergence result of abstract numerical $1$-step methods for the approximation of the solution $y$ on partitions of the interval $[0,T]$ (see, for example, \cite[Section~10.3]{Kress}). It is the basis for proving Theorem~\ref{TaylorConvergent} in Section~\ref{SectionTaylorMethod}. We restrict our attention to equidistant partitions of step size $h > 0$.

\begin{theorem}\label{ConvergenceTheoremGeneral}
Let $K \Subset D$ be a compact neighborhood with $y([0,T]) \subset \mathring{K}$, and let
$$
\Phi : [0,h_0] \times K \to \R
$$
be continuous, $h_0 > 0$. Assume:
\begin{itemize}
\item \emph{Consistency:} $\Phi(0,y) = y$ for all $y \in K$, and $\frac{\partial\Phi}{\partial h} : (0,h_0) \times K \to \R$ exists and extends to a continuous function on $[0,h_0] \times K$ such that $\frac{\partial\Phi}{\partial h}(0,y) = f(y)$ for all $y \in K$.
\item \emph{Lipschitz Condition:} The function $\frac{\partial \Phi}{\partial h} : [0,h_0] \times K \to \R$ satisfies a Lipschitz condition with respect to $y$, i.e., there exists a constant $L > 0$ such that
$$
\Bigl| \frac{\partial\Phi}{\partial h}(h,y) - \frac{\partial\Phi}{\partial h}(h,y') \Bigr| \leq L|y-y'|
$$
for all $0 \leq h \leq h_0$ and $y,y' \in K$. 
\item \emph{Local Truncation Error:} There exists $p \geq 1$ and a constant $C \geq 0$ independent of $0 \leq h \leq h_0$ and $0 \leq t \leq T$ such that
$$
\bigl|y(t+h) - \Phi(h,y(t))\bigr| \leq C h^{p+1}
$$
whenever $0 \leq t+h \leq T$.
\end{itemize}
Then $\Phi$ yields a $1$-step method of order $p$ for the approximation of $y$ on $[0,T]$, i.e., there exist $N_0 \in \N$ and a constant $M \geq 0$ such that for all $N \geq N_0$, $h=\frac{T}{N}$, the following holds:

The sequence of numbers $y_0^{(N)},\ldots,y_{N}^{(N)}$ defined via
$$
\left\{
\begin{aligned}
y_0^{(N)} &= y_0 \\
y_{n+1}^{(N)} &= \Phi\bigl(h,y_n^{(N)}\bigr), \quad n=0,\ldots,N-1,
\end{aligned}
\right.
$$
is well-defined, all $y_n^{(N)} \in \mathring{K}$, and the \emph{Global Error} satisfies
\begin{equation}\label{GlobalErrorOrderp}
\max\limits_{n=0}^N\bigl| y(nh) - y_n^{(N)}\bigr| \leq M h^{p}.
\end{equation}
A valid choice for the constant in \eqref{GlobalErrorOrderp} is $M = \frac{C}{L}\bigl(e^{LT}-1\bigr)$.
\end{theorem}


\section{Differential equations depending on parameters}\label{ODEsDependingOnParameters}

\noindent
Let $\Lambda \subset \R^q$ be open, and $V \subset \R$ be an open interval with $0 \in V$. Suppose $F(w;\lambda)$ is continuously differentiable with respect to the variables $(w;\lambda) \in V \times \Lambda$. Consider the family of ordinary differential equations
\begin{equation}\label{ODEwithparameter}
\left\{
\begin{aligned}
\frac{\partial w}{\partial t}(t;\lambda) &= F(w(t;\lambda);\lambda) \\
w(0;\lambda) &= 0
\end{aligned}
\right.
\end{equation}
for the unknown function $t \mapsto w(t;\lambda)$ depending on the parameter $\lambda \in \Lambda$. The following holds (see \cite{Walter}).

\begin{theorem}\label{ODEParameters}
For each $\lambda \in \Lambda$ there exists a unique maximally extended solution
$$
w(\cdot\,;\lambda) : (t_{\min}(\lambda),t_{\max}(\lambda)) \to V
$$
to \eqref{ODEwithparameter}, where $-\infty \leq t_{\min}(\lambda) < 0 < t_{\max}(\lambda) \leq \infty$.

The functions $t_{\max}, t_{\min} : \Lambda \to \R\cup\{\pm\infty\}$ are lower and upper semicontinuous, respectively, and the set
$$
U_{\max} = \{(t,\lambda);\; \lambda \in \Lambda, \; t_{\min}(\lambda) < t < t_{\max}(\lambda)\} \subset \R\times\R^q
$$
is open. The solution $w$ to \eqref{ODEwithparameter} defines a map $U_{\max} \to V$, and both $w$ and $\frac{\partial w}{\partial t}$ are continuously differentiable in $U_{\max}$. The partial derivatives of $w$ satisfy
\begin{equation}\label{PartialDerivativeFormulas}
\begin{aligned}
\frac{\partial w}{\partial t}(t;\lambda) &= F(w(t;\lambda);\lambda) \quad \textup{(this is just \eqref{ODEwithparameter})}, \\
(\nabla_{\lambda} w)(t;\lambda) &= \int_0^t e^{\int_s^t F_w(w(u;\lambda);\lambda)du}(\nabla_{\lambda}F)(w(s;\lambda);\lambda)\,ds.
\end{aligned}
\end{equation}
In particular, if $F$ is more than once continuously differentiable, then so is $w$, and formulas for higher partial derivatives of $w$ follow from \eqref{PartialDerivativeFormulas} with the Chain Rule.
\end{theorem}

\begin{remark}\label{RemarkODEParameters}
The upper and lower semicontinuity of the endpoint functions of the maximal existence interval follow from the openness of $U_{\max}$. Semicontinuity implies that $t_{\min}$ attains its maximum value $t_{\min}(K) \in \R\cup\{-\infty\}$ and $t_{\max}$ attains its minimum value $t_{\max}(K) \in \R\cup\{\infty\}$ on every compact subset $K \Subset \Lambda$. In particular, $w(t;\lambda)$ is defined (and differentiable) for all $(t;\lambda) \in (t_{\min}(K),t_{\max}(K))\times K$. Thus, for every compact subset $K \Subset \Lambda$, we are guaranteed that $w(t;\lambda)$ exists on $[0,T]\times K$ for some $T > 0$ (depending on $K$). We make use of this in the theoretical Section~\ref{SectionTaylorMethod} of this paper.
\end{remark}


\section{MATLAB source code}\label{SourceCode}

\subsection*{Main program of the quadratic Taylor method}

{\ }

{\tiny
\begin{verbatim}
% Instructions:
%
% 1) Users have to choose a zero tolerance. Any values that are in magnitude
%    less than that value are numerically zero.
% 2) Users have to specify the objective ODE and the initial condition.
% 3) Users have to specify the stepsize. The program currently supports only
%    equidistant time-stepping.
% 4) Users have to specify a viewing window [0,T] in time and [ymin,ymax] for
%    the observed range of values of the solution.
% 5) Users have to assign whether to carry out integrity checks on the stepsize
%    for the method at run-time at each step, or determine a valid stepsize
%    prior to running the program. This is the purpose of boolean apriori:
%    apriori=true means no integrity checks at runtime.
%    apriori=false means the program will check the integrity of the stepsize
%    at runtime at each step
% 6) The program calls the function Arccot, provided separately. Arccot is the
%    real inverse cotangent function with range (0,pi).

%
% User specifications:
%

% Set zero tolerance
zero=1e-14;
% Use symbolic engine for y
syms y 
% Define objective ODE
dydt=exp(y);
% Set initial value
y0=2;
% Set stepsize
h=0.01;
% Set viewing window
ymin=0;
ymax=2.01;
T=2; %time interval [0,T]
% Assign true or false to apriori
apriori=false;


% Initialize the row matrix t. t contains the time steps.
t=0:h:T; 
% Initialize the row matrix yNumerical. yNumerical will later contain the
% numerical approximations for the solution on the time grid.
yNumerical=zeros(1,length(t));
% Store the initial value into the matrix yNumerical
yNumerical(1)=y0;  
for i=1:length(t)-1 
    a=double(subs(diff(dydt,2),y,yNumerical(i))/2);  
    b=double(subs(diff(dydt,1),y,yNumerical(i)));
    c=double(subs(dydt,y,yNumerical(i)));
    delta=b^2-4*a*c;
    stabilityCheck=2-h*b;
    if apriori==true
       if delta>=4*zero
          w=2*c*sinh(sqrt(delta)*h/2)/(sqrt(delta)*cosh(sqrt(delta)*h/2)-b*sinh(sqrt(delta)*h/2));
       elseif delta<=-4*zero
          w=2*c*sin(sqrt(-delta)*h/2)/(sqrt(-delta)*cos(sqrt(-delta)*h/2)-b*sin(sqrt(-delta)*h/2));
       elseif abs(delta)<4*zero
          w=2*c*h/(2-b*h)-h^3*c*delta/(3*(2-b*h)^2);
       end  
    elseif apriori==false
           if stabilityCheck<sqrt(zero)
              warning('Method requires a smaller stepsize in order to be stable.')
              t=t(1,1:i);
              yNumerical=yNumerical(1,1:i);
              break
           else
              if delta>=4*zero
                 if sqrt(delta)<b
                    hmax=(1/sqrt(delta))*log((b+sqrt(delta))/(b-sqrt(delta)));
                 else
                    hmax=inf;
                 end
                 if h<hmax
                    w=2*c*sinh(sqrt(delta)*h/2)/(sqrt(delta)*cosh(sqrt(delta)*h/2)-b*sinh(sqrt(delta)*h/2));
                 else
                    t=t(1,1:i);
                    yNumerical=yNumerical(1:1:i); 
                    Warn=['The algorithm terminates at step ',num2str(i-1),', approximating the solution on
                                the interval [0,',num2str((i-1)*h),'], because the method becomes undefined in
                                the next step due to the chosen stepsize. Suggestion: Rerun the program with
                                smaller stepsize'];
                    warning(Warn)
                    break
                 end
              elseif delta<=-4*zero
                     hmax=(2/sqrt(-delta))*Arccot(b/sqrt(-delta));
                     if h<hmax
                        w=2*c*sin(sqrt(-delta)*h/2)/(sqrt(-delta)*cos(sqrt(-delta)*h/2)-b*sin(sqrt(-delta)*h/2));
                     else
                        Warn=['The algorithm terminates at step ',num2str(i-1),', approximating the solution
                                    on the interval [0,',num2str((i-1)*h),'], because the method becomes
                                    undefined in the next step due to the chosen stepsize.
                                    Suggestion: Rerun the program with smaller stepsize'];
                        warning(Warn)
                        t=t(1,1:i);
                        yNumerical=yNumerical(1,1:i); 
                        break
                     end
              elseif abs(delta)<4*zero
                     w=2*c*h/(2-b*h)-h^3*c*delta/(3*(2-b*h)^2);
              end  
           end
    end
    yassume=w+yNumerical(i);
    if yassume>ymax||yassume<ymin
       Warn=['The algorithm terminates at step ',num2str(i-1),', approximating the solution on the interval
                   [0,',num2str((i-1)*h),'], because the approximate solution leaves the tracking window in the
                   next step. Suggestion: Rerun the program with larger y-viewing window.'];
       warning(Warn)
       t=t(1,1:i);
       yNumerical=yNumerical(1,1:i); 
       break
    else 
       yNumerical(i+1)=yassume;
    end   
end
\end{verbatim}
}

\medskip

\noindent
The function Arccot (required by main program):

{\tiny
\begin{verbatim}
function value=Arccot(x)
if x>=0
    value=acot(x);
else
    value=acot(x)+pi;
end
\end{verbatim}
}

\subsection*{Program that performs step size integrity check a priori}

{\ }

{\tiny
\begin{verbatim}
% Program for apriori check of hmax
%
% Instructions:
%
% 1) Users have to choose a zero tolerance. Any values that are in magnitude
%    less than that value are numerically zero.
% 2) Users have to specify the objective ODE.
% 3) Users have to specify a viewing window [0,T] in time and [ymin,ymax] for
%    the observed range of values of the solution.
% 4) Program requires the function MAX, provided separately.

%
% User specifications:
%

% Set zero tolerance
zero=1e-14;
% Use symbolic engine for y
syms y
% Define objective ODE
dydt=exp(y);
% Set viewing window
ymin=0;
ymax=5;
T=5; %time interval [0,T] 

a=diff(dydt,2)/2;  
b=diff(dydt,1);
c=dydt;
delta=b^2-4*a*c;
s=b^2+abs(delta);
% The separate function MAX is using the standard matlab function 'fminbnd'
bmax=MAX(b,ymin,ymax);
smax=MAX(s,ymin,ymax);
if (smax>zero)&&(bmax>zero)
    hmax=min([2/sqrt(smax) (2-zero)/bmax T]);
elseif  (smax>zero)&&(bmax<=zero)
    hmax=min([2/sqrt(smax) T]);
else
    hmax=T;
end
hmaxOUT=['Suggest stepsize to be less than ',num2str(hmax)];
disp(hmaxOUT)
\end{verbatim}
}

\medskip

\noindent
The function MAX (required for a priori integrity checks on the step size):

{\tiny
\begin{verbatim}
function maxvalue=MAX(Function,leftBound,rightBound)
syms y
NewFunction=Function*(-1);
min=fminbnd(matlabFunction(NewFunction),leftBound,rightBound);
maxInTheMiddle=min*(-1);
valueAtLeftEndpoint=double(subs(Function,y,leftBound));
valueAtRightEndpoint=double(subs(Function,y,rightBound));
Compare=[maxInTheMiddle valueAtLeftEndpoint valueAtRightEndpoint];
maxvalue=max(Compare);
end
\end{verbatim}
}


\end{appendix}




\begin{thebibliography}{99}

\bibitem{BogackiShampine}
P.~Bogacki and L.F.~Shampine, \emph{A 3(2) pair of Runge-Kutta formulas}, Appl.~Math.~Lett. \textbf{2} (1989), 321--325.

\bibitem{Butcher}
J.C.~Butcher, \emph{Numerical Methods for Ordinary Differential Equations}, Second Edition, John Wiley \& Sons, Chichester, 2008.

\bibitem{Corlessetal}
R.~Corless, G.~Gonnet, D.~Hare, D.~Jeffrey, and D.~Knuth, \emph{On the Lambert $W$ function}, Adv.~Comput.~Math. \textbf{5} (1996), 329--359.

\bibitem{HairerNorsettWanner}
E.~Hairer, S.P.~N{\o}rsett, and G.~Wanner, \emph{Solving Ordinary Differential Equations I}, Second Revised Edition, Springer-Verlag, Berlin, 1993.

\bibitem{HochbruckOstermann}
M.~Hochbruck and A.~Ostermann, \emph{Exponential integrators}, Acta Numerica \textbf{19} (2010), 209--286.

\bibitem{KassamTrefethen}
A.-K.~Kassam and L.~Trefethen, \emph{Fourth-order time-stepping for stiff PDEs}, SIAM J.~Sci.~Comput. \textbf{26} (2005), 1214--1233.

\bibitem{Kress}
R.~Kress, \emph{Numerical Analysis}, Graduate Texts in Mathematics, vol.~181, Springer-Verlag, New York, 1998.

\bibitem{Mickens94}
R.E.~Mickens, \emph{Nonstandard finite difference models of differential equations}, World Scientific, River Edge, NJ, 1994.

\bibitem{Mickens00}
R.E.~ Mickens (Ed.), \emph{Applications of nonstandard finite difference schemes (Atlanta, GA, 1999)}, World Scientific, River Edge, NJ, 2000.

\bibitem{Moler2003}
C.~Moler, \emph{Stiff Differential Equations}, Cleve's Corner blog post on the MathWorks web page, 2003. See
https://www.mathworks.com/company/newsletters/articles/stiff-differential-equations.html.

\bibitem{Patidar}
K.C.~Patidar, \emph{On the use of nonstandard finite difference methods}, J.~Difference Equ.~Appl. \textbf{11} (2005), 735--758.

\bibitem{Ralston}
A.~Ralston, \emph{A First Coure in Numerical Analysis}, McGraw-Hill, New York, 1965.

\bibitem{ShampineReichelt}
L.F.~Shampine and M.W.~Reichelt, \emph{The MATLAB ODE suite}, SIAM J.~Sci.~Comput. \textbf{18} (1997), 1--22.

\bibitem{Thieme}
H.R.~Thieme, \emph{Mathematics in Population Biology}, Princeton University Press, Princeton, NJ, 2003.

\bibitem{VigoAguiarRamos}
J.~Vigo-Aguiar and H.~Ramos, \emph{A numerical ODE solver that preserves the fixed points and their stability}, J.~Comput.~Appl.~Math. \textbf{235} (2011), 1856--1867.

\bibitem{Walter}
W.~Walter, \emph{Ordinary Differential Equations}, Springer-Verlag, New York, 1998.

\end{thebibliography}
\end{document}